\documentclass{article}
\usepackage[utf8]{inputenc}
\usepackage{amsmath}
\usepackage{amsthm}
\usepackage{amssymb}
\usepackage{tikz,pgfplots}
\pgfplotsset{compat=1.11}
\usepgfplotslibrary{fillbetween}
\usetikzlibrary{arrows, automata, patterns, intersections}
\usepackage{hyperref}
\usepackage{cite}
\usepackage{authblk}
\usepackage{graphicx}
\usepackage{adjustbox}
\usepackage{caption}
\usepackage{subcaption}
\usepackage{appendix}

\newtheorem{definition}{Definition}
\newtheorem{question}{Question}

\newtheorem{theorem}{Theorem}
\newtheorem{lemma}{Lemma}

\title{Cycle Intersection Graphs and Minimum Decycling Sets of Even Graphs}
\author[1]{Michael Cary\footnote{macary@mail.wvu.edu}}
\affil{Department of Computer Science and Center for Alternative Fuels, Engines, and Emissions, West Virginia University}
\begin{document}
\maketitle
\begin{abstract}
We introduce the cycle intersection graph of a graph, an adaptation of the cycle graph of a graph, and use the structure of these graphs to prove an upper bound for the decycling number of all even graphs. This bound is shown to be significantly better when an even graph admits a cycle decomposition in which any two cycles intersect in at most one vertex. Links between the cycle rank of the cycle intersection graph of an even graph and the decycling number of the even graph itself are found. The problem of choosing an ideal cycle decomposition is addressed and is presented as an optimization problem over the space of cycle decompositions of even graphs.
\end{abstract}



\section{Introduction}
Decycling graphs, the process of deleting vertices from a graph until the graph becomes acyclic, is the vertex analog of the notion of cycle rank. While there is an extensive literature dedicated to studying the cycle rank of graphs, decycling graphs is a relatively newer topic. The topic of decycling graphs also known as a feedback vertex set, primarily focuses on studying decycling sets in order to find the decycling number of a graph. While the cycle rank of a graph $G$ with $c$ components has an explicit formula, $|E(G)|-|V(G)|-c$, finding the decycling number of a graph has been proven to be NP-hard \cite{hartnell2008,karp1972}. Before continuing, we provide formal definitions of decycling sets and the decycling number of a graph.

\begin{definition}\label{d1}
Let $G$ be a graph and $S\subset V(G)$ be a subset of the vertex set. If $G\setminus S$ is acyclic, then $S$ is said to be a decycling set of $G$.
\end{definition}
\begin{definition}\label{d2}
The decycling number of a graph, $\nabla(G)$, is the size of a smallest decycling set.
\end{definition}

As stated, computing and bounding the decycling number is the primary focus of the study of decycling graphs. One of the key contributions of this paper will ultimately be to link the notions of cycle rank and decyling number in the case of even graphs. Even graphs are graphs in which every vertex has even degree. We introduce the cycle intersection graph of a graph, which is defined as follows.

\begin{definition}\label{d3}
Let $G$ be a graph. The cycle intersection graph of $G$ is the graph $CI(G)$ whose vertex set is the set of cycles in a given cycle decomposition of $G$ and in which the edge set is a mapping of the unique intersections between cycles.
\end{definition}

As can be seen from the definition of a cycle intersection graph, even graphs are the only family of graphs whose entire edge set is characterized by their cycle intersection graph. However, spanning even subgraphs are a common and highly useful tool in many areas on graph theory and properties of their existence and structure in graphs and digraphs are well known \cite{cary2017,catlin1988}. The intention of developing this approach to find decyling sets of even graphs is that it can immediately serve as a useful tool for finding decycling sets of other graphs, possibly via either maximum spanning even subgraphs or minimum containing even graphs.

The motivation for introducing cycle intersection graphs of graphs stems from a similar concept that has been used on occasion in graph theory, namely the cycle graph of a graph. The cycle graph of a graph similarly represents cycles in the original graph as vertices, but only joins two vertices if their corresponding cycles share an edge. Our concept is essentially an edge-disjoint analog of this previous concept. For more on cycle graphs of graphs and some of their applications in the literature, the reader may be interested in any of \cite{egawa1991,lopez1995,tan1993,van1992}.

Before detailing the structure of this paper, we take the time to familiarize the reader with previous results on decycling graphs. The very first results on the decycling number of graphs, produced in the seminal paper on the topic, bounded the decycling number of hypercubes \cite{beineke1997}. Improving the bound on the decycling number of hypercubes was continued in \cite{bau2001} and \cite{pike2003}. This body of work inspired research on the decycling number of other practical network structures \cite{wang2015,zdeborova2016}. In more traditional graph theoretic directions, there has been progress on bounding the decycling of regular graph in general \cite{punnim2012}, as well as for cubic graphs specifically \cite{punnim2005,punnim2007} including finding the exact decycling number of the (generalized) Petersen graph \cite{gao2015}. The decycling number of random regular graphs was studied in \cite{bau2002}. Avoiding assumptions of regularity, \cite{levit2011} proved results linking the alternating number of independent sets and the decycling number of a graph, and \cite{salavatipour2006} studied this problem by linking the decycling number of graphs to the problem of finding large spanning forests of graphs in the case of planar graphs. The decycling number of the Cartesian product of a graph and a complete graph was proven in \cite{hartnell2008}. For a survey of results on decycling graphs, the reader is referred to \cite{bau2007}.

\section{Cycle Intersection Graphs and Decycling Sets}
We begin this work by using the notion of a cycle intersection graph to generate a specific decycling set of the original graph. Our focus will be restricted to even graphs until the very last section of this paper, and we note now that all graphs are assumed to be connected, for otherwise one may simple sum the decycling number of each connected component in order to find the decycling number of a disconnected graph. This section in particular will study the special case of even graphs which admit a unique cycle decomposition, or, as we will find out is equivalent, those even graphs whose cycle intersection graphs are trees.

\begin{theorem}\label{t1}
$\nabla(G)\leq|E(CI(G))|$.
\end{theorem}
\begin{proof}
Let $G$ be a connected even graph, let $CI(G)$ be its cycle intersection graph, and let $f: E(CI(G))\rightarrow V(G)$. Since each edge in $CI(G)$ represents a vertex in $V(G)$, we may let $S\subset V(G)$ be the image of $f$. Clearly $G\setminus S$ is disconnected as no two cycles remain adjacent in $G$ and every cycle of $G$ has at least one vertex removed (else $G$ was not connected). 
\end{proof}
Do note that the above proof does not imply that the map $f$ is injective. Several cycles may all intersect at one vertex. In an extremal case, the cycle intersection graph could be a complete graph on $n>>1$ vertices, yet all of the represented cycles share a single, common vertex, in which case the decycling number of the graph would be one.

Returning our attention to bounding the decycling number of even graphs, one potential place to start searching for an upper bound on the decycling number of even graphs would be the size of a minimum cycle decomposition of the graph. But as it turns out, at least in some cases, this bound can actually be improve upon.

When considering the cycle intersection graph of an even graph, one possibility is that the cycle intersection graph is a tree. Whenever this is the case, the previous theorem tells us that there is a decycling set strictly smaller than the size of a smallest cycle decomposition. In fact, the following theorem gives an explicit value for the decycling number of even graphs whose cycle intersection graphs are trees. First, however, we prove some necessary lemmas.

\begin{lemma}\label{l1}
Let $G$ be an even graph and $CI(G)$ its cycle intersection graph. If $CI(G)$ is a tree, then $G$ has a unique cycle decomposition.
\end{lemma}
\begin{proof}
The proof is by induction on the size of a smallest cycle decomposition. If $G$ consists of a single cycle, then the proof is trivial, so assume that $CI(G)$ has more than one vertex. Let $\mathcal{F}(G)$ be a smallest possible cycle decomposition of $G$ (with respect to the cardinality of $\mathcal{F}(G)$) and let $f:\mathcal{F}(G)\rightarrow V(CI(G))$ be an isomorphism between sets. Let $v\in V(CI(G))$ be a leaf in $CI(G)$ and let $C_{v}$ be the pre-image of $v$ under $f$. Consider the graph $G^{\prime}=G\setminus C_{v}$. By our inductive hypothesis, $G^{\prime}$ has a unique cycle decomposition, call it $\mathcal{F}(G^{\prime})$. By seeing that $C_{v}\cap G^{\prime}=\{v\}$ we see that $C_{v}$ must be a member of every cycle decomposition of $G$, hence $\mathcal{G}=\mathcal{F}(G^{\prime})\cup C_{v}$ is the only cycle decomposition of $G$ and the proof is complete. 
\end{proof}

The careful reader will notice that Theorem \ref{t1} is actually a corollary to this lemma. In fact,  turns out that we can actually improve upon this result with only slightly more effort. To do this, we must first introduce a notation. Let $CI(G)$ be the cycle intersection graph of an even graph. We denote a minimum spanning forest of $CI(G)$ by $MSF(CI(G)))$. Additionally, by the size of a minimum spanning forest, $|MSF(CS(AG))|$, we mean the total number of edges and isolated vertices of the spanning forest, i.e., $|MSF(CI(G))|=|E(MSF(CI(G)))|+|\{v\in V(MSF(CI(G)))\ |\ d(v)=0\}|$.

\begin{lemma}\label{l2}
Let $G$ be an even graph and $CI(G)$ its cycle intersection graph. If $CI(G)$ is a tree, then $\nabla(G)\leq |MSF(CI(G))|$.
\end{lemma}
\begin{proof}
$MSF(CI(G))$ covers every vertex of $CI(G)$ hence removes a vertex from each cycle in $G$. By the previous lemma, the set of vertices of $G$ corresponding to the edges of $MSF(CI(G))$ constitutes a decycling set of $G$.
\end{proof}

\begin{lemma}\label{l3}
Let $G$ be an even graph and $CI(G)$ its cycle intersection graph. If $CI(G)$ is a tree, then any minimum decycling set of $G$ constitutes a spanning forest of $CI(G)$.
\end{lemma}
\begin{proof}
First, notice that an isolated vertex is a valid member of a spanning forest. Since a decycling set contains a vertex from every cycle of $G$, and since $CI(G)$ is a tree, the result follows.
\end{proof}

Finally, we are ready to present a significant result. While the ensuing theorem is a direct consequence of the previous lemma, we distinguish these two as the previous lemma is a useful and significant result in its own right.

\begin{theorem}\label{t2}
Let $G$ be an even graph and $CI(G)$ its cycle intersection graph. If $CI(G)$ is a tree, then $\nabla(G)=|MSF(CI(G))|$.
\end{theorem}
\begin{proof}
The result follows from the previous lemma and from the fact that if $CI(G)$ is a tree, then $CI(G)$ is the unique cycle intersection graph of $G$ as $G$ has a unique decomposition into cycles.
\end{proof}

The next series of results are a direct application of Theorem \ref{t2} to even graphs whose cycle intersection graphs are well-defined trees.

\begin{theorem}\label{t3}
Let $G$ be an even graph and $CI(G)$ its cycle intersection graph. If $CI(G)$ is a path and $|E(CI(G))|=n$, then the decycling number of $G$ is given by $\nabla(G)=\big\lceil\frac{n+1}{2}\big\rceil$.
\end{theorem}
\begin{proof}
If there are an odd number of edges in $CI(G)$ then the size of a minimum spanning forest is the size of a maximum matching. If there are an even number of edges in $CI(G)$ then a maximum matching does not saturate one vertex, hence a minimum spanning forest in this case requires an additional vertex. The result follows from the size of a maximum matching in a path.
\end{proof}

\begin{theorem}\label{t4}
Let $G$ be an even graph and $CI(G)$ its cycle intersection graph. If $CI(G)$ is a star and $|E(CI(G))|=n$, then the decyling number of $G$ is given by $\nabla(G)=n$.
\end{theorem}
\begin{proof}
It suffices to observe that the size of a minimum spanning forest of a star is $n$.
\end{proof}

\section{When $CI(G)$ is simple but not acyclic}
Up to this point we have only considered even graphs whose cycle intersection graphs are trees. To generalize this, we obviously need to consider even graphs whose cycle intersection graphs are cyclic. Before doing this, we first need to take the time to discuss the importance of choosing an ideal cycle decomposition of an even graph.

Now, the previous approach, which relied on spanning forests, seems only to be useful whenever $CI(G)$ is simple. If $G$ is an even graph, but $CI(G)$ is a multigraph, a spanning forest of $CI(G)$ (which in this case would be any edge) clearly does not lead up to a decycling set of $G$. To illustrate this point, consider the following example presented below in Figure \ref{fig1}. Let $G$ be an even graph whose cycle intersection graph is a multigraph on two vertices and $n$ edges. In this example we have apparently chosen to decompose the even graph $G$ into only two cycles that intersect at precisely three vertices. However, we may easily define a new cycle decomposition of $G$ in which there are $n$ cycles, in which each cycle has precisely two intersection vertices, and in which any two cycles intersect in at most one vertex; i.e., $CI(G)$ is itself a cycle. 

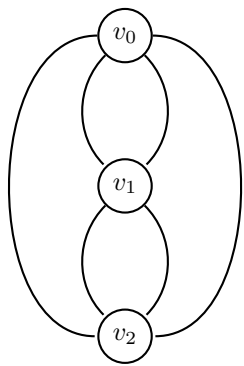
\begin{figure}[h!]
\centering
\begin{tikzpicture}[-,>=stealth',shorten >=1pt,auto,node distance=2cm,
                    thick,main node/.style={circle,draw}]
  \node[main node] (A)					{$v_{0}$};
  \node[main node] (B) [below of=A]		{$v_{1}$};
  \node[main node] (C) [below of=B] 	{$v_{2}$};
  \draw[thick,-] [bend left=45] (A) to (B);
  \draw[thick,-] [bend right=45] (A) to (B);
  \draw[thick,-] [bend left=45] (B) to (C);
  \draw[thick,-] [bend right=45] (B) to (C);
  \draw[thick,-] [bend left=90] (A) to (C);
  \draw[thick,-] [bend right=90] (A) to (C);
\end{tikzpicture}
\caption{An example of an even graph that could be represented as a multigraph on two vertices and three edges by decomposing the graph into a two edge-disjoint cycles which start at the vertex $v_{0}$ and travelling though the vertices $v_{1}$ and $v_{2}$ before returning to $v_{0}$, i.e., two edge-disjoint cycles that intersect in three vertices. Notice that this even graph can also decompose into cycles characterized by the vertex pairs $(v_{0},v_{1})$, $(v_{1},v_{2})$, and $(v_{0},v_{2})$, respectively. In the latter decomposition, the cycle intersection graph is a three cycle.}
\label{fig1}
\end{figure}

We see that by removing any two of the vertices $v_{0}$, $v_{1}$, or $v_{2}$, the original even graph is successfully decycled. Furthermore, notice that any smallest possible spanning forest of the simple version of $CI(G)$ in the previous example constitutes a smallest possible decycling set of $G$, while a smallest spanning forest of the multigraph version of $CI(G)$ does not correspond to a subset of the vertex set which suffices to decycle $G$.

Now that the importance of choosing an appropriate cycle decomposition (and thus cycle intersection graph) has been established, we proceed by proving results on the decycling number of even graph via simple cycle intersection graphs of these even graphs. We obtain these results through means of a crucial link between the decycling number of an even graph and the cycle rank of its cycle intersection graph. First, though, what would have been a very convenient property of even graphs (and was very nearly proposed herein as a conjecture), is the claim that every even graph admits a cycle decomposition whose associated cycle intersection graph is simple. The following (rather obvious) example proves otherwise.

\begin{figure}[h!]
\centering
\begin{tikzpicture}[-,>=stealth',shorten >=1pt,auto,node distance=2cm,
                    thick,main node/.style={circle,draw}]
  \node[main node] (A)                      {};
  \node[main node] (B) [below left of=A]    {};
  \node[main node] (C) [below right of=A]   {};
  \node[main node] (D) [right of=C]         {};
  \node[main node] (E) [below right of=B]   {};
  \draw[thick,-] (A) to (B);
  \draw[thick,-] (A) to (C);
  \draw[thick,-] (A) to (E);
  \draw[thick,-] (B) to (E);
  \draw[thick,-] (C) to (E);
  \draw[thick,-] (A) to (D);
  \draw[thick,-] (D) to (E);
\end{tikzpicture}
\caption{An example of an even graph that does admit a single cycle decomposition whose associated cycle intersection graph is simple.}
\label{fig2}
\end{figure}
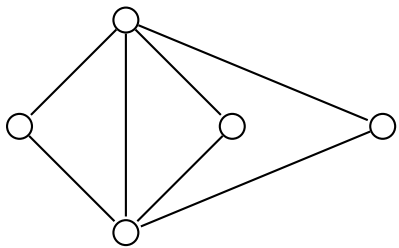

With this counterexample, we are forced to accept that not every even graph admits a simple cycle intersection graph. However, we use the rest of this section to prove results on those even graphs which do have this convenient, non-universal property.

\begin{theorem}\label{t5}
Let $G$ be an even graph which admits a simple cycle intersection graph $CI(G)$. Then $\nabla(G)\leq |E(CI(G))|-|V(CI(G))|+1-|MSF(CI(G))|$, i.e., the decycling number of an even graph $G$ is equal to the number of edges needed to decycle $CI(G)$ and then remove a minimum spanning forest from the resulting acyclic subgraph of $CI(G)$.
\end{theorem}
\begin{proof}
Let $S_{CI(G)}$ be any subset of $E(CI(G))$ whose removal from $CI(G)$ leaves some spanning tree of $CI(G)$, and let $S_{G}$ be the vertices of $G$ corresponding to the edges of $S_{CI(G)}$. It suffices to prove that no edge of $G^{\prime}=G\setminus S_{G}$ belongs to more than one cycle in $G^{\prime}$ (irrespective of the chosen cycle decomposition and cycle intersection graph of $G$), because then we may let $\hat{G^{\prime}}$ be the graph obtained by removing all paths of $G$ not contained in any cycle and apply Theorem \ref{t2} to complete the proof.

To derive a contradiction, assume that the edge $e\in E(G^{\prime})$ is contained in two different cycles in $G^{\prime}$. Let $C_{0}$ be the cycle in the chosen cycle decomposition of $G$ to which $e$ belongs. Then any other cycle in $G^{\prime}$ which contains $e$ must be formed by a series of intersecting cycles which returns to the cycle $C_{0}$. Denote the cycles comprising this cycle chain $C_{0}, C_{1}, \dots, C_{k-1}$ (we assume the cycle chain has length $k$) and denote the set of intersection vertices between these cycles by $IV=\{v_{0,1}, v_{1,2}, \dots, v_{k-1,0}\}$. In order for $e$ to belong to a cycle other than $C_{0}$ that still exists in $G^{\prime}$, it must be the case that $VI\subset V(G^{\prime})$. But this contradicts the definition of the sets $S_{CI(G)}$ and $S_{G}$. Therefore no edge in $E(G^{\prime})$ is a member of multiple cycles in $G^{\prime}$ and we may apply Theorem \ref{t2} to establish our upper bound and complete the proof.
\end{proof}

It is crucially important to note that this is left as an inequality in part because there may be multiple cycle decompositions of an even graph which admit a simple cycle intersection graph, and these cycle decompositions may not all have the same size. Even if every cycle decomposition of a given even graph which leads to a simple cycle intersection graph has the same size, the question remains, can the inequality proven in Theorem \ref{t5} be reduced to equality in every case? Additionally, if an even graph has cycle decompositions of different sizes which both yield simple cycle intersection graphs, which cycle decomposition is optimal? These questions will be recounted at the end of this paper, but we next turn our attention to those even graphs which do not admit a single simple cycle intersection graph.

\section{Cycle Intersection Multigraphs}
In this section we consider those even graphs which do not admit a single simple cycle intersection graph, i.e., every cycle decomposition of the even graph corresponds to a cycle intersection \textit{multigraph}. The primary issue with the technique introduced to find decycling sets of even graphs which admits simple cycle intersection graphs, when applied to even graphs which admit cycle intersection multigraphs, is a consequence of the multiplicity of edges between adjacent vertices in the cycle intersection multigraph. Let $G$ be an even graph and see that any two cycles of our even graph which intersect at multiple vertices, i.e., any pair of vertices in the cycle intersection multigraph which share multiple edges, may not be successfully decycled by removing only one intersection vertex. However, it turns out that we can generalize the theorem proven in the previous section to multigraphs, thereby attaining a bound on the decycling number of all even graphs. Intuitively this result holds true because, in the process of making the cycle intersection graph acyclic, we also remove all digons (2-cycles) in the cycle intersection graph.

\begin{theorem}\label{t6}
Let $G$ be an even graph and let $CI(G)$ be a cycle intersection graph of $G$ (not necessarily simple). Then $\nabla(G)\leq |E(CI(G))|-|V(CI(G))|+1-|MSF(CI(G))|$.
\end{theorem}
\begin{proof}
The proof is similar to that of Theorem \ref{t5}, requiring the removal of all but one edge from each set of multiple edges to remove all digons from $CI(G)$.
\end{proof}

The most prominent implication of this result is that the choice of cycle decomposition truly matters. By minimizing the number of multiple edges, we reduce the number of vertices we add to the decycling set of $G$ we construct from $CI(G)$. However, as there is not necessarily an ideal cycle decomposition of a given even graph to choose for this process, it may be best to consider this an optimization problem. For a given even graph $G$, let $\mathcal{C}_{G}$ be the space of all cycle decompositions of $G$, let $CI_{C}(G)$ denote the cycle intersection graph associated to the specific cycle decomposition $C$ of $G$, and let $r(G)$ denote the cycle rank of a graph. The ideal cycle decomposition of an even graph, for the purpose of finding the decycling number of that even graph, is any solution of the form
\begin{equation}\label{stateOP}
\hat{C}\in\mathcal{C}_{G}\ \mathrm{s.t.}\ r(CI_{\hat{C}}(G))\leq r(CI_{C}(G))\ \forall\ C\in\mathcal{C}(G)
\end{equation}

\section{Conclusion and Open Problems}
In this paper we introduce the notion of a cycle intersection graph of a graph which is the edge-disjoint analog of the cycle adjacency graph of a graph. We showed how the cycle intersection graph is obtained from an even graph and linked the cycle intersection graphs of an even graph directly to the cycle decompositions of an even graph. We then used this tool to bound (and in some cases explicitly compute) the decycling number of even graphs. However, the problem of finding an ideal cycle decomposition is difficult and was formalized in Statement \ref{stateOP}. Nevertheless, by using either (or both) a maximum spanning even subgraph or a minimum covering even graph of a graph, we can bound the decycling number of any graph in terms of these two associated even graphs. Given these results, a few interesting questions can be asked.

\begin{question}\label{q1}
Which even graphs admit a simple cycle intersection graph?
\end{question}

\begin{question}\label{q2}
What does a solution to Statement \ref{stateOP} look like? How can it/they be found? What is the complexity of finding a solution to Statement \ref{stateOP} for an arbitrary even graph?
\end{question}

\begin{question}\label{q3}
Can the inequality proven in Theorem \ref{t6} be reduced to equality in general?
\end{question}

\bibliography{Decycling1}
\end{document}